\documentclass[a4, 12pt, 
]{amsart}
\usepackage{mathtools}
\usepackage{enumerate}
\usepackage{amssymb}
\usepackage{amsmath}
\usepackage{amsthm}
\usepackage{enumitem}
\usepackage{esint}
\setlist[enumerate,2]{
  ref=\alph*,
}
\usepackage{hyperref}
\usepackage{braket}
\hypersetup{
           breaklinks=true,   
           colorlinks=true,   
        }
\numberwithin{equation}{section}
\theoremstyle{plain}
\newtheorem{thm}[equation]{Theorem}

\newtheorem{prop}[equation]{Proposition}
\newtheorem{lem}[equation]{Lemma}

\newtheorem*{claim*}{Claim}
\theoremstyle{definition}
\newtheorem{defn}[equation]{Definition}

\newtheorem{notation}[equation]{Notation}
\theoremstyle{remark}
\newtheorem{rem}[equation]{Remark}

\usepackage{xparse}
\usepackage{bbm}
\usepackage{todonotes}
\usepackage{comment}
\author[R.\ Murakami]{Rei Murakami}
\address{Mathematical Institute, Tohoku University, 6-3, Aramaki Aza-Aoba, Aoba-ku, Sendai 980-8578, Japan}
\email{rei.murakami.p3@dc.tohoku.ac.jp, reimurakami66@gmail.com}

\begin{document}

\title[An analytic proof of Griffiths' conjecture on compact Riemann surfaces]{An analytic proof of Griffiths' conjecture on compact Riemann surfaces}

\begin{abstract}
    Griffiths' conjecture asserts that a holomorphic vector bundle is ample if and only if it admits a Hermitian metric with positive curvature.
    In this paper, we present a new proof of this conjecture on compact Riemann surfaces using a system of PDEs introduced by Demailly. 
    Our argument combines techniques developed by Uhlenbeck–Yau for Hermitian–Einstein metrics with Pingali's reduction of the problem to an a priori estimate.
\end{abstract}

\maketitle

\section{Introduction}
Given a holomorphic vector bundle $E$ on a compact complex manifold $X$, Griffiths conjectured that differential-geometric and algebro-geometric notions of positivity of $E$ are equivalent. 
The differential-geometric notion of positivity is called \textit{Griffiths positivity} and defined in terms of the Chern curvature.
Explicitly, the curvature $\sqrt{-1}F_h$ of a Hermitian metric $h$ is Griffiths positive at a point $p \in X$ if
$$\sqrt{-1}(F_h)_{i\bar{j}\,\alpha\bar{\beta}}\, \xi^i \bar{\xi}^j v^\alpha \bar{v}^\beta > 0$$
for all nonzero $\xi \in E_p$ and  $v \in T^{1,0}_pX$.
The algebro-geometric notion of positivity for $E$ is defined via the ampleness of the tautological line bundle $\mathcal{O}_{\mathbb{P}(E^*)}(1)$.
The implication from Griffiths positivity to ampleness is well known, and for line bundles the equivalence itself follows from the Kodaira embedding theorem.

In his recent work \cite{Dem21}, Demailly proposed a new analytic approach to this conjecture. He introduced a system of PDEs, which we refer to as \textit{the Demailly system}, and conjectured that the existence of solutions characterizes ampleness in terms of Griffiths positivity. 
This approach has already been tested in the case of the vortex bundle \cite{Man23}.

In this paper, we study the Demailly system on Riemann surfaces. Although Griffiths' conjecture was already settled in this case by \cite{Ume, CF} and \cite[Theorem 1.10]{LZZ}, we give a new proof using Demailly's analytic framework.
Given a reference Hermitian metric $h_{\mathrm{ref}}$ on $E$, we write any Hermitian metric $h$ as $h=e^{-f}g h_\mathrm{ref}$, where $f$ is a smooth function and $g \in \Gamma(\mathrm{End}\,E)$ is a positive Hermitian matrix with respect to $h_\mathrm{ref}$ satisfying $\det g=1$. 
Then, the Demailly system on Riemann surfaces takes the form
\begin{equation}\label{eq: Dem}
    \begin{cases}
        \det\big(\sqrt{-1}F_h/\omega_0+(1-t)\alpha\big)=e^{\lambda f} a_0 \\
        \Lambda_{\omega_0}\sqrt{-1}F_h^\circ=-e^f \ln g,
    \end{cases}
\end{equation}
where a positive number $\alpha$, a smooth positive function $a_0\in C^\infty(X)$, and a Kähler form $\omega_0$ are chosen so that the system admits a solution $h_0$ at $t=0$ with $f_0=0$ and $\sqrt{-1}F_{h_0}/\omega_0+\alpha>0$ (see the proof of \cite[Theorem 3.1]{UY}). Here, we denote by $\sqrt{-1}F_h^\circ$ the trace-free part of $\sqrt{-1}F_h$. Also, hereafter, we suppress $\otimes\mathrm{Id}_E$, so $(1-t)\alpha$ in \eqref{eq: Dem} denotes $(1-t)\alpha \otimes\mathrm{Id}_E$.
The constants $\lambda$ and $\alpha$ are chosen to ensure local uniqueness at $t=0$ \cite[Propositions 2.1 and 2.2]{Pin23}. The first equation of \eqref{eq: Dem} is formulated to detect the positivity of the matrix on the left-hand side. The second equation is attached to overcome the undeterminancy of the equation. The specific form is chosen to use the study of the (cushioned) Hermitian-Einstein equation in \cite{UY}.

Our main result is the following:

\begin{thm}\label{thm: main}
    Let $X$ be a compact Riemann surface and $E$ a holomorphic vector bundle. Then, the following are equivalent:
    \begin{enumerate}[font=\normalfont]
        \item\label{item: solv} The Demailly system \eqref{eq: Dem} is solvable at $t=1$ with a Hermitian metric whose curvature is Griffiths positive.
        \item\label{item: Gposi} $E$ admits a Hermitian metric whose curvature is Griffiths positive.
        \item\label{item: ample} $E$ is ample.
    \end{enumerate}
\end{thm}

The implication \eqref{item: solv}$\Rightarrow$\eqref{item: Gposi} follows trivially and
the implication \eqref{item: Gposi}$\Rightarrow$\eqref{item: ample} follows easily. Hence, the essential part of the theorem is to show \eqref{item: ample}$\Rightarrow$\eqref{item: solv}.
Pingali \cite{Pin23} previously studied the Demailly system on Riemann surfaces and established Theorem \ref{thm: main} for a sum of ample line bundles. For general vector bundles, using the Leray–Schauder degree theory \cite[Theorem 1.2]{Pin23}, he reduced the problem to an a priori lower bound for $f$, which is the focus of this paper (Proposition \ref{prop: f}).

The paper is organized as follows. In Section \ref{sec: curv}, we establish an a priori estimate for the curvature $\sqrt{-1}F_h$ (Proposition \ref{prop: keyest}). In Section \ref{sec: comp}, using techniques from \cite{UY}, we show that if $f$ were unbounded below, one could construct a quotient sheaf $Q$ with $\deg Q \le 0$, contradicting the ampleness of $E$ (since ampleness implies $\deg Q > 0$ for every quotient $Q$). This yields the desired lower bound for $f$.
Finally, in Section \ref{sec: alt}, for direct sums of line bundles (as in \cite{Pin23}), we provide an alternative proof of \eqref{item: Gposi} $\Rightarrow$ \eqref{item: solv} by obtaining a lower bound for $f$ by the maximum principle, which partially simplifies the proof as discussed in \cite[Remark 3.1]{Pin23}. The restriction to this special case is due to the lack of Lemma \ref{lem: osc} for general vector bundles.
We end this paper with some remarks on further study in Section \ref{sec: rem}.

\subsection*{Acknowledgements}
The author thanks his supervisor, Shin-ichi Matsumura, for introducing him to the topic of this paper and for his constant support. He is also grateful to Satoshi Jinnouchi for introducing him to the book \cite{LT}, to Kuang-Ru Wu for giving valuable comments on the first version of this paper, and to Vamsi Pingali for kindly answering his questions about the paper \cite{Pin23} and for giving valuable comments, as well as for pointing out a subtle issue regarding the smoothness of eigenvalues in the first version of this paper. He also thanks the anonymous referee for helpful comments and for introducing him to the paper \cite{LZZ}, and for pointing out a subtle issue regarding the smoothness of eigenvalues in the first version of this paper.
This work was supported by JSPS KAKENHI Grant Number JP24KJ0346.

\section{Proof}
\subsection{Curvature estimates}\label{sec: curv}

In this section, we prove an a priori estimate for the curvature in the Demailly system (Proposition \ref{prop: keyest}). We freely use the following notation: 

\begin{notation}\label{notation}
\begin{itemize}
    \item
    We fix $h_0$ as a reference Hermitian metric on $E$ and denote the covariant derivative and the curvature associated with $h_0$ by $\partial_0$ and $\sqrt{-1}F_0$.
    \item  We denote the eigenvalue functions of $\ln g$ by $\lambda_1\ge\lambda_2\ge\dots\ge\lambda_r$, which are continuous but not necessarily smooth. We also denote by $\lambda_\mathrm{max}$ the largest eigenvalue function $\lambda_1$.
    \item 
    Define the open dense set $W\subset X$ by
    \begin{equation*}
        W:=\{x\in X\mid \text{the multiplicities of $\{\lambda_i\}$ are locally constant at $x$}\}.
    \end{equation*}    
    Given a point $p\in W$, each $\lambda_i$ is smooth around $p$, and we locally diagonalize $\ln g\in \Gamma(\mathrm{End}E)$ on a neighborhood $U$ of $p$ via a local frame $\{s_i\}_{i=1}^r$.
    Moreover, since $g$ is Hermitian with respect to $h_0$, we can choose such a local frame $\{s_i\}_{i=1}^r$ which is also orthonormal with respect to $h_0$ on $U$ (see \cite[Proposition 7.4.5]{LT}).
    \item 
    For a point $p\in X\setminus W$, we locally approximate $g$ by the following way so that the eigenvalues $\{\hat{\lambda}_i\}$ of the approximation $\ln\hat{g}$ are smooth. Choose a local frame $\{\tilde{s}_i\}_{i=1}^r$, orthonormal with respect to $h_0$ at $p$, so that $\ln g$ is diagonalized at $p$. Let $$B:=\sum b_i \tilde{s}_i\otimes \tilde{s}^i,$$ where $b_i\in C^\infty$ satisfy $0\ge b_1> b_2>\dots > b_r$ and $b_1(p)=0$ and are sufficiently small. Define $$\ln \hat{g}=\ln g+B.$$ Then the eigenvalues $\{\hat{\lambda}_i\}_{i=1}^r$ of $\ln \hat{g}$ are distinct at $p$. Hence we diagonalize $\ln \hat{g}$ on a neighborhood $U$ of $p$ via a local frame $\{\hat{s}_i\}_{i=1}^r$ as before. 
    Since $\ln\hat{g}$ is diagonal at $p$ also with respect to $\{\tilde{s}_i\}$, we may arrange $\hat{s}_i(p)=\tilde{s}_i(p)$. Also, by definition, we have $\lambda_1(p)=\hat{\lambda}_1(p)$ and $\lambda_i\ge\hat{\lambda}_i$. Letting $b_i\to 0$ in $C^\infty(U)$ yields $\hat{g}\to g$ in $C^\infty(U)$.
    
    \item 
    The Demailly system \eqref{eq: Dem} is expressed as
    \begin{equation}\label{eq: Dem'}
    \begin{cases}
        \det \left(\dfrac{\beta}{r}+\Delta f-e^f\ln g+(1-t)\alpha\right)=e^{\lambda f}a_0,\\
        \Lambda_{\omega_0}\sqrt{-1}F_h^\circ=-e^f \ln g,
    \end{cases}
    \end{equation}
    where $\beta$ denotes $\Lambda_{\omega_0}\mathrm{tr}\, \sqrt{-1}F_0$ and $\Delta$ denotes the Laplacian associated to $\omega_0$. 
    Here, we suppress $\otimes\mathrm{Id}_E$, so $\frac{\beta}{r}+\Delta f+(1-t)\alpha$ denotes $(\frac{\beta}{r}+\Delta f+(1-t)\alpha)\otimes\mathrm{Id}_E$.

    \item In this paper, a positive constant $C$ denotes a constant which is independent of $t\in[0,1]$, and may vary from line to line.
\end{itemize}
\end{notation}

With the localization on $U$ stated in Notation \ref{notation}, we have
\begin{align*}
    g=\sum_{i=1}^r e^{\lambda_i}\, s_i\otimes s^i
\end{align*}
on a neighborhood $U$ of $p\in W$. Define a $(1,0)$-form $A^{(1,0)}$ and a $(0,1)$-form $A^{(0,1)}$ by
\begin{align*}
    \partial_0 s_i=\big(A^{(1,0)}\big)^j_i s_j, \quad \bar{\partial} s_i=\big(A^{(0,1)}\big)^j_i s_j,
\end{align*}
where $\partial_0$ is the $(1,0)$-part of the covariant derivative associated with the reference Hermitian metric $h_0$. Here we adopt the Einstein summation convention and omit the summation symbol.

The following lemma follows from straightforward computations, but since we could not find it in the literature, we record it here for completeness.
\begin{lem}\label{lem: laplace}
    We have
    \begin{equation}\label{eq: laplace}
    \begin{aligned}
        \Delta\lambda_i=-\big(\Lambda_{\omega_0}\sqrt{-1}\bar{\partial}(g^{-1}\partial_0g)\big)^i_i+\sum_{j=1}^r\big(e^{\lambda_i-\lambda_j}-1\big)C_{i,j}
    -\sum_{j=1}^r\big(e^{\lambda_j-\lambda_i}-1\big)C_{j,i} \quad \text{on $W$},
    \end{aligned}
    \end{equation}
    where  $C_{i,j}$ are nonnegative functions defined by 
    $$C_{i,j}:=-\Lambda_{\omega_0}\sqrt{-1}\left(\big(A^{(1,0)}\big)^j_i \wedge\big(A^{(0,1)}\big)^i_j\right).$$
\end{lem}

\begin{proof}
We see that
\begin{align*}
    \partial_0 g=\sum_{i=1}^r \partial \big(e^{\lambda_i}\big) s_i\otimes s^i + \sum_{i\neq j} \big(e^{\lambda_j}-e^{\lambda_i}\big)\big(A^{(1,0)}\big)^i_j s_i\otimes s^j,
\end{align*}
since $\partial_0 s^j=-\big(A^{(1,0)}\big)^j_i s^i$. Consequently, we have
\begin{align*}
    \big(g^{-1}\partial_0  g\big)=\sum_{i,j,k}\big(g^{-1}\big)^i_k\big(\partial_0g\big)^k_j s_i\otimes s^j=\sum_{i=1}^r (\partial\lambda_i) s_i\otimes s^i+\sum_{i\neq j}\big(e^{\lambda_j-\lambda_i}-1\big)\big(A^{(1,0)}\big)^i_j s_i\otimes s^j.
\end{align*}
Therefore, we have
\begin{equation*}
\begin{aligned}
    \left(\Lambda_{\omega_0}\sqrt{-1}\bar{\partial}(g^{-1}\partial_0g)\right)^i_i
    =&- \Delta\lambda_i
    -\sum_{j\neq i}\big(e^{\lambda_i-\lambda_j}-1\big)\Lambda_{\omega_0}\sqrt{-1}\left(\big(A^{(1,0)}\big)^j_i \wedge\big(A^{(0,1)}\big)^i_j\right)\\
    &+\sum_{j\neq i}\big(e^{\lambda_j-\lambda_i}-1\big)\Lambda_{\omega_0}\sqrt{-1}\left(\big(A^{(1,0)}\big)^i_j \wedge\big(A^{(0,1)}\big)^j_i\right).
\end{aligned}
\end{equation*}
Finally, since $\{s_i\}$ is orthonormal with respect to $h_0$ on $U$, we have
\begin{align*}
    0=\partial (h_0)_{i\bar{j}}=\big(A^{(1,0)}\big)^k_i (h_0)_{k\bar{j}}+(h_0)_{i\bar{k}}\overline{\big(A^{(0,1)}\big)^k_j}=\big(A^{(1,0)}\big)^j_i+\overline{\big(A^{(0,1)}\big)^i_j},
\end{align*}
which implies that $C_{i,j}$ are nonnegative.
\end{proof}

\begin{lem}\label{lem: Deltalambda}
    Let $(f,g)$ be a solution of \eqref{eq: Dem}. Then, the function $\lambda_{\mathrm{max}}$ is $C\omega_0$-subharmonic for some constant $C$. In particular, there exists a constant $C$ such that $\fint_X (\sup \lambda_{\mathrm{max}}-\lambda_{\mathrm{max}})\le C$, where $\fint$ denotes the integral with respect to $\omega_0$ divided by the total volume $\int_X\omega_0$.
\end{lem}    

\begin{proof}
    
    Fix a point $p\in X$. Let $q$ be a $C^2$ function defined on a neighborhood of $p$ such that $q\ge\lambda_{\mathrm{max}}$ and $q(p)=(\lambda_{\mathrm{max}})(p)$. 
    If $p\in W$, then using Lemma \ref{lem: laplace}, we see that
    \begin{equation*}
    \begin{aligned}
        \Delta q
        &\ge \Delta\lambda_1\ge -\big(\Lambda_{\omega_0}\sqrt{-1}\bar{\partial}(g^{-1}\partial_0g)\big)^1_1\\
        &=e^f\lambda_1-\big(\Lambda_{\omega_0}\sqrt{-1}F^\circ_0\big)^1_1\ge -C \quad \text{at } p,
    \end{aligned}
    \end{equation*}
    where we used the second equation of \eqref{eq: Dem'} in the equality, and the last inequality follows from $\lambda_1\ge0$, obtained by $\operatorname{tr}\ln g=0$. If $p\notin W$, we have $q\ge\lambda_\mathrm{max}\ge\hat{\lambda}_1$ on $U$ and $q(p)=\lambda_\mathrm{max}(p)=\hat{\lambda}_1(p)$. Using Lemma \ref{lem: laplace} for $\hat{\lambda}_1$, we see that
    \begin{equation*}
        \begin{aligned}
            \Delta q
            &\ge \Delta\hat{\lambda}_1\ge -\big(\Lambda_{\omega_0}\sqrt{-1}\bar{\partial}(\hat{g}^{-1}\partial_0\hat{g})\big)^1_1\ge-\big(\Lambda_{\omega_0}\sqrt{-1}\bar{\partial}(g^{-1}\partial_0g)\big)^1_1-C\\
            &=e^f\lambda_1-\big(\Lambda_{\omega_0}\sqrt{-1}F^\circ_0\big)^1_1-C\ge -C \quad \text{at } p,
        \end{aligned}
    \end{equation*}
    where the third inequality follows since $\hat{g}$ can approximate $g$ in $C^\infty$.
    
    Hence, we obtain $\Delta \lambda_{\mathrm{max}}\ge-C$ in the viscosity sense, which is equivalent to say that $\lambda_{\mathrm{max}}$ is $C\omega_0$-subharmonic. The last assertion is standard, and therefore we omit it here (e.g. \cite[Theorem 1.46]{GZ}).
    \end{proof}

\begin{prop}\label{prop: keyest}
    There exists a constant $C$ such that $e^f \lambda_{\mathrm{max}} <C$ for a solution $(f,g)$ of \eqref{eq: Dem}. In particular, we have $|\Delta f|\le C$.
\end{prop}

\begin{proof}
    We see that $\Delta f\ge -C$ for some constant $C$ by taking the trace of the inequality
    \begin{equation}\label{eq: curvposi}
        \frac{\beta}{r}+\Delta f-e^f\ln g+(1-t)\alpha>0
    \end{equation} 
    since $\mathrm{tr}\ln g=0$.
    Also, the upper estimate of $\Delta f$ follows immediately from the estimate of $e^f |\ln g|$, since the right-hand side of the first equation of \eqref{eq: Dem} is bounded by an upper estimate of $f$, which is obtained from the maximum principle \cite[Remark 2.1]{Pin23}.
    Moreover, since $\operatorname{tr} \ln g=0$, we have $-\lambda_r=\sum_{i\neq r}\lambda_i\le (r-1)\lambda_1$. In particular, we have
    $$e^f |\ln g|=e^f \sqrt{\sum\lambda_i^2}\le C e^f\lambda_1.$$
    Thus, it suffices to obtain an estimate of $e^f\lambda_{\mathrm{max}}$. 
    We have
    \begin{equation}\label{eq: int}
        \begin{aligned}
            \fint_X e^f\lambda_{\mathrm{max}}
            =&\fint_X e^f (\lambda_{\mathrm{max}}-\sup \lambda_{\mathrm{max}})+e^f\sup \lambda_{\mathrm{max}}\\
            \ge& -C+ e^{\sup f}\sup \lambda_{\mathrm{max}}\fint_Xe^{f-\sup f}\\
            \ge& -C+ e^{\sup f}\sup \lambda_{\mathrm{max}}\, e^{\fint_X (f-\sup f)} \\
            \ge& -C+C e^{\sup f} \sup \lambda_{\mathrm{max}}.
        \end{aligned}
    \end{equation}
    Here, we used Lemma \ref{lem: Deltalambda} and an upper estimate of $f$ \cite[Remark 2.1]{Pin23} in the first inequality, the convexity of the exponential in the second inequality and $\Delta f\ge -C$, which implies $\fint f-\sup f\ge -C$ by Green's representation formula, in the last inequality.  On the other hand, using inequality \eqref{eq: curvposi}, we also see that
    \begin{equation}\label{eq: L^1}
        \fint_X e^f \lambda_{\mathrm{max}}< \fint_X \frac{\beta}{r}+\Delta f+(1-t)\alpha=\fint_X \frac{\beta}{r}+(1-t)\alpha\le C.
    \end{equation}
    Combining this with \eqref{eq: int}, we obtain $e^f\lambda_{\mathrm{max}}\le e^{\sup f}\sup \lambda_{\mathrm{max}}\le C$.
    \end{proof}

    \begin{rem}
        We can also prove Proposition \ref{prop: keyest} using $\sqrt{|\ln g|^2+1}$ instead of $\lambda_{\mathrm{max}}$. More precisely, we have
        \begin{equation}\label{eq: deltanorm}
        \begin{aligned}
                -\Delta\sqrt{|\ln g|^2+1}
                =&\Lambda_{\omega_0}\sqrt{-1}\bar{\partial}\sum_i\frac{\lambda_i\partial\lambda_i}{\sqrt{|\ln g|^2+1}}=\Lambda_{\omega_0}\sqrt{-1}\bar{\partial}\left(g^{-1}\partial_0 g,\frac{\ln g}{\sqrt{|\ln g|^2+1}}\right)\\
                =& \left(\Lambda_{\omega_0}\sqrt{-1}\bar{\partial}\left(g^{-1}\partial_0g\right),\frac{\ln g}{\sqrt{|\ln g|^2+1}}\right)-\Lambda_{\omega_0}\left(g^{-1}\partial_0g,\partial_0\frac{\ln g}{\sqrt{|\ln g|^2+1}}\right)\\
                \le&-e^f\sqrt{|\ln g|^2+1}+C,
        \end{aligned}
        \end{equation}
        where the inequality follows from the second equation of the Demailly system \eqref{eq: Dem'} and the inequality
        \begin{equation*}
            \begin{aligned}
                &\Lambda_{\omega_0}\sqrt{-1}\left(g^{-1}\partial_0g,\partial_0\frac{\ln g}{\sqrt{|\ln g|^2+1}}\right)\\
                =&\Lambda_{\omega_0}\sqrt{-1}\sum_{i=1}^r\partial\lambda_i\wedge\bar{\partial}\left(\frac{\lambda_i}{\sqrt{|\ln g|^2+1}}\right)+\sum_{i\neq j}(e^{\lambda_j-\lambda_i}-1)\frac{\lambda_j-\lambda_i}{\sqrt{|\ln g|^2+1}}\left|\big(A^{(1,0)}\big)^i_j\right|^2\\
                \ge&\frac{1}{\sqrt{|\ln g|^2+1}}\left(\sum_{i=1}^r |\partial\lambda_i|^2-\frac{\big|\sum_i\lambda_i\partial\lambda_i\big|^2}{\sum_i\lambda_i^2+1}\right)
                \ge\frac{1}{(|\ln g|^2+1)^\frac{3}{2}}\left(\sum_{i\neq j}\left|\lambda_i\partial\lambda_j-\lambda_j\partial\lambda_i\right|^2\right)\ge0.
            \end{aligned}
        \end{equation*}
        Therefore, we have $\Delta\sqrt{|\ln g|^2+1}\ge -C$, which implies an inequality similar to \eqref{eq: int}. By integrating \eqref{eq: deltanorm}, we obtain an $L^1$-bound of $e^f\sqrt{|\ln g|^2+1}$, which completes the proof.
    \end{rem}

\subsection{Construction of a nonpositive quotient sheaf}\label{sec: comp}
We complete the proof of Theorem \ref{thm: main} by constructing a quotient sheaf $Q$ of $E$ such that $\deg Q\le0$ if $f$ is not bounded from below. This is achieved by the technique developed in \cite{UY}. First, recall the definition of a weakly holomorphic subbundle and its relation to a coherent subsheaf.

\begin{defn}[{\cite[Definition 3.4.2]{LT}}]
    An element $\pi\in L^2_1(\mathrm{End}\, E)$ is called a weakly holomorphic subbundle of $E$ if in $L^1(\mathrm{End}\, E)$ we have
    \begin{equation}\label{eq: weakholbdle}
        \pi^*=\pi=\pi^2 \quad \mathrm{and} \quad (\mathrm{Id}_E-\pi)\circ\bar{\partial}(\pi)=0.
    \end{equation}
\end{defn}

\begin{thm}[{\cite[Theorem 3.4.3]{LT}, \cite[S290-292]{UY}}]\label{thm: weakhol}
    A weakly holomorphic subbundle $\pi$ of $E$ represents a coherent subsheaf $\mathcal{F}$ of $E$, and an analytic subset $S\subset X$ with the following properties:
    \begin{enumerate}[font=\normalfont]
        \item $\mathrm{codim}_XS\ge2$
        \item\label{item: reg} $\pi|_{X\setminus S}$ is $C^\infty$ and satisfies \eqref{eq: weakholbdle}.
        \item \label{item: subbdl}$\mathcal{F}':=\mathcal{F}|_{X\setminus S}=\pi|_{X\setminus S}(E|_{X\setminus S})\hookrightarrow E|_{X\setminus S}$ is a holomorphic subbundle.
    \end{enumerate}
\end{thm}

Now, let us define $m_t:=\sup_X (\lambda_{\mathrm{max}})_t$ and $\tilde{g}_t:=e^{-m_t} g_t$,
where $(\lambda_{\mathrm{max}})_t(p)$ is the largest eigenvalue of $\ln g_t$ at $p\in X$.
The arguments in \cite[Lemma 4.1 and Proposition 4.1]{UY} also work in our case by replacing $\varepsilon$ with $e^f$. Therefore, we have the following:

\begin{lem}\label{lem: L^2_1(LT)}
    For a sequence of times $\{t_i\}\subset[0,1]$ such that $t_i\to t_\infty$, 
    there is a subsequence, still denoted by $\{t_i\}$, such that $t_i\to t_\infty$ and the sequence $\tilde{g}_i:=\tilde{g}_{t_i}$ has the following properties:
    \begin{enumerate}[font=\normalfont]
        \item\label{item: L^2_1} For $\sigma\in(0,1]$, the sequence $\tilde{g}^\sigma_i$ weakly converges in $L^2_1$ to $\tilde{g}^\sigma_\infty\neq 0$. Here, for a Hermitian endomorphism $g$, a Hermitian endomorphism $g^\sigma$ is defined as
        $$g^\sigma:=\sum_{i=1}^r e^{\sigma\lambda_i}s_i\otimes s^i$$
        via the localization in Notation \ref{notation}.
        \item\label{item: sigmaj} There is a sequence of $\sigma_j\in(0,1]$ such that $\sigma_j$ converges to zero and the sequence $\tilde{g}_\infty^{\sigma_j}$ weakly converges in $L^2_1$ to $\tilde{g}^0_\infty$.
        \item\label{item: pi} $\pi:=\mathrm{Id}_E-\tilde{g}^0_\infty$ is a weakly holomorphic subbundle of $E$.
    \end{enumerate}
\end{lem}
            
According to Theorem \ref{thm: weakhol}, we have a coherent subsheaf $\mathcal{F}\subset E$ corresponding to $\pi$ defined in Lemma \ref{lem: L^2_1(LT)} \eqref{item: pi}. 
Although an analytic subset $S$ in Theorem \ref{thm: weakhol} is empty since $X$ is one-dimensional, hereafter we compute everything in the way that works in general dimensions as in \cite{UY}.
Since $\tilde{g}^\sigma_i$ also converges strongly in $L^2$ to $\tilde{g}^\sigma_\infty\neq 0$ and $\tilde{g}^\sigma_i\le \mathrm{Id}_E$, the limit $\tilde{g}^0_\infty$ has a strictly positive eigenvalue almost everywhere. 
Therefore, we have

\begin{equation}\label{eq: rank}
    \operatorname{rank}\mathcal{F}=\operatorname{rank} \operatorname{Im}\pi=\operatorname{rank} (\mathrm{Id}_E-\tilde{g}_\infty^0)\le r-1.
\end{equation}

We denote the quotient sheaf $E/\mathcal{F}$ by $Q$. By \eqref{eq: rank}, we have $Q\neq 0$. Also, since $g_t$ is $h_0$-Hermitian, the limit $\pi$ is also $h_0$-Hermitian almost everywhere. Hence, we have $\mathcal{F'}^\perp=\operatorname{Im}\pi^\perp$, where $\mathcal{F'}$ is a subbundle defined in Theorem \ref{thm: weakhol} \eqref{item: subbdl} and $\pi^\perp:=(\mathrm{Id}_E-\pi)$.

\begin{prop}\label{prop: f}
    Suppose that there exists a sequence of points $\{p_i\}\subset X$ and times $\{t_i\}\subset[0,1]$ such that $f_i(p_i)\to -\infty$ as $i\to\infty$ for a solution $(f_t,g_t)$ of the Demailly system \eqref{eq: Dem'}, where $f_i:=f_{t_i}$. Then, we have $\deg {Q}\le0$.
\end{prop}

\begin{proof}
    We induce the Hermitian metric $h_Q$ on $Q|_{X\setminus S}$, which is a quotient bundle, from the reference Hermitian metric $h_0$. 
    Then, we have 
    \begin{equation*}
        \Lambda_{\omega_0}\operatorname{tr}\sqrt{-1}F_{h_Q}= \pi^{\perp}\circ\Lambda_{\omega_0}\operatorname{tr}\sqrt{-1}F_0\circ\pi^{\perp}+|\pi^\perp\partial_0\pi|^2 \quad \text{on } X\setminus S,
    \end{equation*}
    where $\pi^\perp\partial_0\pi$ is the second fundamental form. Note that by the same argument as in \cite[equation (4.15)]{UY}, we have $\pi\partial_0\pi=0$, and therefore $\pi^\perp\partial_0\pi=\partial_0\pi$, which is in $L^2$.  
    Thus, we have
    \begin{equation}\label{eq: degQ}
    \begin{aligned}
        \deg Q
        =&\int_{X\setminus S} \big(\mathrm{tr}\sqrt{-1}F_{h_Q}\big)=\int_{X\setminus S} \mathrm{tr} \left(\sqrt{-1}F_0\circ\pi^{\perp}\right)+|\partial_0\pi|^2\omega_0\\
        =&\lim_{j\to\infty}\lim_{i\to\infty}\int_X\mathrm{tr}\Big(\sqrt{-1}F_0\circ \tilde{g}^{\sigma_j}_i\Big)+\int_X|\partial_0\pi|^2\omega_0\\
        =&\lim_{j\to\infty}\lim_{i\to\infty}\int_{X} \mathrm{tr}\left(\big(\sqrt{-1}F_{h_i}+\sqrt{-1}\bar{\partial}\partial f_i-\sqrt{-1}\bar{\partial}(g_i^{-1}\partial_0g_i)\big)\circ\tilde{g}^{\sigma_j}_i\right)+\int_X|\partial_0\pi|^2\omega_0\\
        =&\lim_{j\to\infty}\lim_{i\to\infty}\int_{X}\mathrm{tr}\big(\sqrt{-1}F_{h_i}\circ\tilde{g}^{\sigma_j}_i\big)+\lim_{j\to\infty}\lim_{i\to\infty}\int_{X}\sqrt{-1}\bar{\partial}\partial f_i\  (\mathrm{tr}\ \tilde{g}^{\sigma_j}_i)\\
        &+\left(-\lim_{j\to\infty}\lim_{i\to\infty}\int_{X}\mathrm{tr}\Big(\big(\sqrt{-1}\bar{\partial}(g_i^{-1}\partial_0g_i)\circ\tilde{g}^{\sigma_j}_i\big)\Big)+\int_X|\partial_0\pi|^2\omega_0\right),
    \end{aligned}
    \end{equation}
    where the first equality follows from \cite[Proposition (3.4.9)]{LT}, the third equality follows from the fact that $\tilde{g}^{\sigma_j}_i\to \pi^\perp$ strongly in $L^2$, and we denote $h_i:=h_{t_i}$.
    The last term is nonpositive by the proof of inequality (4.18) in \cite{UY}.
    The second term is zero, since Proposition \ref{prop: keyest} implies that 
    \begin{equation*}
        \lim_{j\to\infty}\lim_{i\to\infty}\int_{X}\sqrt{-1}\bar{\partial}\partial f_i\  (\mathrm{tr}\ \tilde{g}^{\sigma_j}_i)=\lim_{i\to\infty}\int_X \sqrt{-1}\bar{\partial}\partial f_i\ (\operatorname{tr}\pi^\perp)
    \end{equation*}
    and the right-hand side is zero since $\mathrm{tr} \ \pi^\perp$ is a constant integer almost everywhere following from $\pi\in L^2_1$. We claim that the first term is nonpositive if $f_i(p_i)\to-\infty$ as $i\to\infty$. Note that, by Proposition \ref{prop: keyest}, we have $\operatorname{osc} f\le C$ and consequently $f_i\to -\infty$ on $X$ if $f_i(p_i)\to-\infty$. Fix an arbitrary point $p\in X$. We have 
    \begin{equation*}
        \mathrm{tr}\big(\Lambda_{\omega_0}\sqrt{-1}F_{h_i}\circ\tilde{g}^{\sigma_j}_i\big)=\sum_{k=1}^r \big(\Lambda_{\omega_0}\sqrt{-1}F_{h_i}\big)^k_k \ e^{\sigma_j(\lambda_k-m)},
    \end{equation*}
    where $\lambda_1\ge\lambda_2\ge\dots\ge\lambda_r$ are eigenvalues of $\ln g_{t_i}$ and $m:=m_{t_i}$ is as defined after Theorem \ref{thm: weakhol}.
    For an integer $k\in\{1,2,\dots,r\}$ such that $\lambda_k-m\to-\infty$ as $i\to\infty$, we have
    \begin{equation}\label{eq: forslowk}
        \big(\Lambda_{\omega_0}\sqrt{-1}F_{h_i}\big)^k_k \ e^{\sigma_j(\lambda_k-m)}\to 0
    \end{equation}
    since $(\sqrt{-1}F_h\big)^k_k$ is uniformly bounded by Proposition \ref{prop: keyest}. On the other hand, assume an integer $k\in\{1,2,\dots,r\}$ satisfies $\lambda_k-m> -C$ for some constant $C$ uniformly in $i$. Then, at $p$, we have
    \begin{equation}\label{eq: diverge}
    \begin{aligned}
        0\le& \big(\Lambda_{\omega_0}\sqrt{-1}F_{h_i}\big)^k_k+(1-t_i)\alpha=\frac{\beta}{r}+\Delta f_i-e^{f_i}\lambda_k+(1-t_i)\alpha\\
        \le&\frac{\beta}{r}+\Delta f_i-e^{f_i}(\lambda_1-C)+(1-t_i)\alpha,
    \end{aligned}
    \end{equation}
    where the last inequality follows from $-\lambda_k(p)\le-m+C\le-\lambda_1(p)+C$. Since $f_i\to-\infty$ on $X$, by the first equation of \eqref{eq: Dem}, we have $\beta/r+\Delta f_i-e^{f_i}\lambda_1+(1-t_i)\alpha\to0$. Applying this to \eqref{eq: diverge} and noting that $\lambda_k-m>-C$ and $\sigma_j\to0$ as in Lemma \ref{lem: L^2_1(LT)} \eqref{item: sigmaj}, we consequently obtain that
    \begin{equation}\label{eq: forblowupk}
         \big(\Lambda_{\omega_0}\sqrt{-1}F_{h_i}\big)^k_k \ e^{\sigma_j(\lambda_k-m)}\to -(1-t_\infty)\alpha\le0 \quad \mathrm{as} \ i\to\infty.
    \end{equation}
    Combining \eqref{eq: forslowk} and \eqref{eq: forblowupk}, we see that $\lim_{j\to\infty}\lim_{i\to\infty}\mathrm{tr}\big(\Lambda_{\omega_0}\sqrt{-1}F_{h_i}\circ\tilde{g}^{\sigma_j}_i\big)\le0$. Since $\sqrt{-1}F_h$ is uniformly bounded by Proposition \ref{prop: keyest} and $\tilde{g}^\sigma$ is uniformly bounded, by the Lebesgue dominated convergence theorem, we have 
    $$\lim_{j\to\infty}\lim_{i\to\infty}\int_{X}\mathrm{tr}\big(\sqrt{-1}F_{h_i}\circ\tilde{g}^{\sigma_j}_i\big)\le0.$$
    Applying these estimates to \eqref{eq: degQ}, we obtain $\deg Q\le 0$.
\end{proof}

\begin{proof}[Proof of Theorem \ref{thm: main}]
    Assume $E$ is ample. Then, any quotient bundle is also ample, and, therefore, its degree is positive. By Proposition \ref{prop: f}, we have a constant $C$ such that $f\ge -C$ for a solution $(f,g)$ of the Demailly system \eqref{eq: Dem'}. By \cite[Theorem 1.2]{Pin23}, the lower estimate of $f$ completes the proof.
\end{proof}

\subsection{An alternative proof of \eqref{item: Gposi}$\Rightarrow$\eqref{item: solv} on a direct sum of line bundles}\label{sec: alt}

\begin{lem}\label{lem: Gposi}
    Suppose that the reference metric $h_0$ has Griffiths positive curvature $\sqrt{-1}F_0$. Then, for a solution $(f,g)$ of \eqref{eq: Dem'}, there exists a constant $C$ such that $f\ge -\mathrm{osc}\, \lambda_{\mathrm{max}}-C$.
\end{lem}

\begin{proof}
    Let $p_0\in X$ be a point such that $(f-\lambda_{\mathrm{max}})(p_0)=\min_X (f-\lambda_{\mathrm{max}})$. As in Notation \ref{notation}, we diagonalize $\ln g$ at $p_0$ via a local frame $\{s_i\}$ if $p_0\in W$ or a local frame $\{\hat{s}_i\}$ if $p_0\notin W$. 
    Note that if $\ln  g$ is diagonalized, so is the curvature $\sqrt{-1}F_h$, since we have
    \begin{equation*}
    \begin{aligned}
        \sqrt{-1}F_h
        =\frac{\operatorname{tr}\sqrt{-1}F_h}{r}+\sqrt{-1}F_h^\circ
        =\frac{\operatorname{tr}\sqrt{-1}F_h}{r}-e^f \ln g\,\omega_0,
    \end{aligned}
    \end{equation*}
    where the last equality is the second equation of the Demailly system \eqref{eq: Dem'}.
    Therefore, the first equation of the Demailly system \eqref{eq: Dem'} is expressed as
    \begin{equation}\label{eq: Demprod}
        \prod_{i=1}^r\left(\frac{(\sqrt{-1}F_0)^i_i}{\omega_0}+\Delta f+\frac{\sqrt{-1}\bar{\partial}(g^{-1}\partial_0g)^i_i}{\omega_0}+(1-t)\alpha\right)=e^{\lambda f}a_0.
    \end{equation}
    Define $a_i:=(\sqrt{-1}F_0)^i_i/\omega_0$, which is positive by assumption. 
    If $p_0\in W$, by inequality \eqref{eq: laplace}, we have
    \begin{equation*}
    \begin{aligned}
       \left(\Lambda_{\omega_0}\sqrt{-1}\bar{\partial}(g^{-1}\partial_0g)\right)^1_1
       =&- \Delta\lambda_1+\sum_{j\neq 1}\big(e^{\lambda_1-\lambda_j}-1\big)C_{1,j}-\sum_{j\neq 1}\big(e^{\lambda_j-\lambda_1}-1\big)C_{j,1}\\
       \ge&- \Delta\lambda_1.
    \end{aligned}
    \end{equation*}
    Then, we have
    \begin{equation*}
        0\le\Delta (f-\lambda_1)\le \Delta f+\sqrt{-1}\bar{\partial}(g^{-1}\partial_0g)^1_1/\omega_0 \quad \text{at } p_0.
    \end{equation*}
    If $p_0\notin W$, applying inequality \eqref{eq: laplace} for $\hat{\lambda}_1$, we have
    \begin{equation*}
    \begin{aligned}
       \varepsilon+\left(\Lambda_{\omega_0}\sqrt{-1}\bar{\partial}(g^{-1}\partial_0g)\right)^1_1
       \ge&\left(\Lambda_{\omega_0}\sqrt{-1}\bar{\partial}(\hat{g}^{-1}\partial_0\hat{g})\right)^1_1
       \ge- \Delta\hat{\lambda}_1,
    \end{aligned}
    \end{equation*}
    where $\varepsilon$ is a small constant chosen later and the first inequality follows since $\hat{g}$ can be arbitrarily close to $g$ in $C^\infty$.
    Since $\hat{\lambda}_1\le\lambda_1$ and $\hat{\lambda}_1(p_0)=\lambda_1(p_0)$, we have
    \begin{equation*}
        0\le\Delta (f-\hat{\lambda}_1)\le \Delta f+\sqrt{-1}\bar{\partial}(g^{-1}\partial_0g)^1_1/\omega_0+\varepsilon \quad \text{at } p_0.
    \end{equation*}
    Therefore, in either case, we have
    \begin{equation*}
    \begin{aligned}
     0<&a_1+(1-t)\alpha\le a_1+\Delta f+\sqrt{-1}\bar{\partial}(g^{-1}\partial_0g)^1_1/\omega_0+(1-t)\alpha+\varepsilon\\
     =&\beta/r+\Delta f-e^f \lambda_1+(1-t)\alpha+\varepsilon\\
     \le&\beta/r+\Delta f-e^f \lambda_i+(1-t)\alpha+\varepsilon\\
     =&a_i+\Delta f+\sqrt{-1}\bar{\partial}(g^{-1}\partial_0g)^i_i/\omega_0+(1-t)\alpha+\varepsilon.
    \end{aligned}
    \end{equation*}
    By applying this estimate to \eqref{eq: Demprod}, we obtain
    \begin{equation*}
        0<(a_1/2)^r<(a_1+(1-t)\alpha-\varepsilon)^r\le e^{\lambda f}a_0 \quad \text{at } p_0,
    \end{equation*}
    where we chose $\varepsilon<a_1/2$ so that the second inequality follows.
    Thus, we obtain $f(p_0)>-C$ for some constant $C$. Hence, we obtain
    \begin{equation*}
        f=f-\lambda_{\mathrm{max}}+\lambda_{\mathrm{max}}\ge-C-\lambda_{\mathrm{max}}(p_0)+\lambda_{\mathrm{max}}>-\mathrm{osc}\, \lambda_{\mathrm{max}}-C,
    \end{equation*}
    which is the desired estimate.
\end{proof}

We show that for a direct sum of line bundles, we have $\operatorname{osc} \lambda_{\mathrm{max}}\le C$.
Let $E$ be a direct sum of line bundles $E=\oplus L_i$ and take a direct sum of Hermitian metrics on each line bundle as a reference metric. Then, considering $\ln g=\oplus u_i$, where $u_i$'s are smooth functions, we can express the Demailly system \eqref{eq: Dem'} as 
\begin{equation}\label{eq: directsum}
    \begin{cases}
       \prod_i \left(\frac{\beta}{r}+\Delta f-e^f u_i+(1-t)\alpha\right)=e^{\lambda f}a_0,\\
        \left(\Lambda_{\omega_0}\sqrt{-1}F_0^\circ\right)^i_i-\Delta u_i=-e^f u_i.
    \end{cases}
\end{equation}
In this setting, we denote $u_{\mathrm{max}}:=\max_i u_i$.

\begin{lem}\label{lem: osc}
    Let $(f,\{u_i\})$ be a solution of \eqref{eq: directsum}. Then, there exists a constant $C$ such that $\operatorname{osc} u_\mathrm{max}\le C$.
\end{lem}

\begin{proof}
   Let $p_0$ be a point such that $\sup u_{\mathrm{max}}=u_{\mathrm{max}}(p_0)$. Then, there exists $i\in \{1,2,\dots,r\}$ such that $\sup u_{\mathrm{max}}=u_i(p_0)$. By Proposition \ref{prop: keyest} (or \cite[Proposition 3.2]{Pin23}), we have $|e^fu_i|\le C$ for some constant $C$. Thus, by the second equation of \eqref{eq: directsum}, we have $|\Delta u_i|\le C$.
   Hence, by Green's representation formula, we have $\operatorname{osc} u_i\le C$. Then, we see that
   $$\sup u_{\mathrm{max}}-u_{\mathrm{max}}\le u_i(p_0)-u_i\le C,$$
   which is the desired estimate. 
\end{proof}

Combining the two lemmas above, we obtain a lower bound of $f$. As a result, by \cite[Theorem 1.2]{Pin23}, we obtain an alternative proof of the implication \eqref{item: Gposi}$\Rightarrow$\eqref{item: solv} in Theorem \ref{thm: main} in the setting of a direct sum of line bundles with a direct sum of Hermitian metrics.

\begin{rem}
    If one can prove that $\operatorname{osc} \lambda_{\mathrm{max}}$ is bounded on vector bundles, then Lemma \ref{lem: Gposi} is valid to give an alternative proof on vector bundles. Not only that, the estimate of $\operatorname{osc} \lambda_{\mathrm{max}}$ also gives an alternative proof of Proposition \ref{prop: keyest}, since at the point $p_0$ where $\sup f=f(p_0)$ we have
    $$e^f\ln g\le \frac{\beta}{r}+\Delta f+(1-t)\alpha\le C,$$
    and therefore in the same way as in the proof of \cite[Proposition 3.2]{Pin23} we have
    $$e^f|\ln g|\le Ce^f\lambda_{\mathrm{max}}\le Ce^{\sup f}(\lambda_{\mathrm{max}}-\lambda_{\mathrm{max}}(p_0)+\lambda_{\mathrm{max}}(p_0))\le C.$$
    In conclusion, the oscillation bound for $\lambda_{\mathrm{max}}$ appears to be a powerful estimate in the study of the Demailly system. However, we do not know if this estimate holds in general. 
\end{rem}

\section{Some remarks on further study}\label{sec: rem}

After \cite{Dem21}, both equations of the Demailly system have been modified for the following reasons. 

The first equation of the Demailly system introduced in \cite{Dem21} is formulated to detect dual Nakano positivity, which is strictly stronger than Griffiths positivity and not efficient for attacking Griffiths' conjecture by \cite{LSY}, as remarked in \cite[Section 1]{Dem22}. To include the case of detecting Griffiths positivity, Demailly introduced the operator $\Phi_P$ \cite[equations (2.1), (2.2), (2.3), and (2.17)]{Dem22} and the equation \cite[equation (3.24)]{Dem22}. It is also new that \cite[equation (3.24)]{Dem22} uses the determinant bundle instead of the additional data $\alpha$ as in \cite{Dem21}, motivated by the results of \cite{Ber, LSY, MT}.

Regarding the second equation, Pingali \cite{Pin21} showed that the solution of the second equation of the Demailly system in \cite[Theorem 2.17]{Dem21} has to be conformal to the Hermitian-Einstein metric if it exists, and this uniqueness is too strong in general to attack Griffiths' conjecture. Variants of the second equation
such as \cite[Section 2.19]{Dem21} or \cite[equation (3.24$^\circ$)]{Dem22} can overcome this problem.
Note that these problems do not occur in dimension one, as shown in this paper, since the Hermitian-Einstein metric on a vector bundle $E$ in dimension one is positively curved if $E$ is ample and (dual) Nakano positivity is equivalent to Griffiths positivity in dimension one. 

To study dual Nakano positivity and generalize the result of this paper to an $n$-dimensional Kähler manifold $X$, we can therefore consider, for example, the following system \cite[equations (2.9), (2.16), and Section 2.19]{Dem21}:
\begin{equation}\label{eq: Demhigh}
    \begin{cases}
        \det_{T_X^{1,0}\otimes E^*}\big((^T\sqrt{-1}F_h)/\omega_0+(1-t)\alpha\otimes\operatorname{Id}_{T_X^{1,0}\otimes E^*}\big)=e^{\lambda f} a_0 \\
        \dfrac{\sqrt{-1}F_h^\circ\wedge\omega_t^{n-1}}{\omega_0^n}=-e^f \ln g\\
        \omega_t:=\dfrac{\operatorname{tr}_E(\sqrt{-1}F_h+(1-t)\alpha\omega_0\otimes\operatorname{Id}_E)}{r\alpha+1},
    \end{cases}
\end{equation}
where $(^T\sqrt{-1}F_h)$ is the transpose of $\sqrt{-1}F_h$ in $\operatorname{End}E$-part and $(^T\sqrt{-1}F_h)/\omega_0\in\Gamma(\operatorname{End}(T^{1,0}_X\otimes E))$ is given locally by 
$$\big((^T\sqrt{-1}F_h)/\omega_0\big)^{i\gamma}_{j\delta}:=(\sqrt{-1}F_h)^{j}_{i\delta\bar{\eta}}\ g_0^{\gamma\bar{\eta}}.$$
Here $g_0^{\gamma\bar{\eta}}$ is a local expression of the inverse $g_0^{-1}$ of the Hermitian metric $g_0$ on $T_X$ corresponding to $\omega_0$. Other notations in \eqref{eq: Demhigh} are the same as in \eqref{eq: Dem}. The difficulty in studying \eqref{eq: Demhigh} is that we cannot plug in the second equation to the first equation to obtain a simpler equation like \eqref{eq: Dem'}. Related to \eqref{eq: Demhigh}, \cite{Pan} recently solved a type of the first equation of \eqref{eq: Demhigh} in the conformal class of a Nakano-positive Hermitian metric (that is, only by changing $f$).   

Finally, we also remark that we do not know if the new system in \cite{Dem22} also gives another proof of Griffiths' conjecture in dimension one.

\end{document}